\documentclass[12pt]{amsart}
\usepackage{graphicx}
\usepackage{amstext}
\usepackage{amsfonts, amssymb}
\usepackage{amsmath}
\usepackage{color}
\usepackage{latexsym}
\usepackage{indentfirst}
\usepackage{amssymb}

\newtheorem{theorem}{Theorem}

\newtheorem{lemma}[theorem]{Lemma}
\newtheorem{corollary}[theorem]{Corollary}

\theoremstyle{definition}
\newtheorem{remark}{Remark}

\newtheorem{problem}{Problem}
\title{Index one minimal surfaces in positively curved $3$-manifolds}
\author{Antonio Ros}     
\thanks{This work is supported in part by the IMAG–Maria de Maeztu grant CEX2020-001105-M / AEI / 10.13039/501100011033, 
MICINN grant PID2020-117868GB-I00 and Junta de Andalucıa grant P18-FR4049.}

\begin{document}


\begin{abstract}
{\small We construct a Riemannian metric of positive sectional curvature on the \mbox{$3$-dimensional} projective space with a two-sided closed embedded minimal surface of 
\mbox{genus $3$,} \mbox{index $1$ and} nullity $0$.}

\end{abstract}


\maketitle

\section{Introduction}

The study of the index of minimal surfaces in Riemannian three manifolds is a topic of 
{interest} in different contexts and in particular in Min-max theory for the area functional. If $M$ 
is a compact orientable Riemannian three-manifold with positive Ricci curvature,
then the index of closed two-sided minimal surfaces cannot be $0$ and Ketover, Marques and Neves
 \cite{ketovermarquesneves} prove that admits a closed embedded two-sided minimal surface $\Sigma$ of index $1$ realizing 
the Heegaard genus of $M$ such that its area 
give the width of the ambient metric. Hamilton, \cite{hamilton}, showed that such a $M$ is diffeomorphic to a spherical space form and 
therefore $Heegaard\, genus(M)\leq 2$. 
On the other hand, from the second variation formula of area it \mbox{follows} that the genus of an index 
$1$ minimal surface $\Sigma$ in $M$ is $\leq 3$.
This bound holds true even for $Ric\geq 0$, Ros \cite{ros1}, and it is known to be sharp for the cubic flat $3$-torus, Ross \cite{ross}. 
The geometric significance of the option $genus(\Sigma)=3$ is not clear  and
its relevance has been noted by Schoen and Neves \cite{neves}.

\vspace{.2cm}

In the case of constant curvature, Viana \cite{viana} proved that there is a positive $c$ such that, 
if $\Sigma\subset M^3$ is a closed $2$-sided embedded minimal surface of index $1$ and genus $3$ in a spherical space form $M$, 
then the order of the fundamental group of $M$ is bounded $|\pi_1(M)|\leq c$. 

\vspace{.2cm}

In this paper, we construct, for the first time, a closed minimal surface of genus $3$ and index $1$ in a positively curved ambient space. 
We prove the following:



\begin{theorem}
There is a Riemannian metric $g$ of positive sectional curvature on the real projective space
${R P}^3$ and a compact orientable embedded minimal surface $\Sigma$ of genus $3$ in $({RP}^3,g)$ with index $1$ and nullity $0$.
\label{teor0}
\end{theorem}

 As a consequence of the property of ${RP}^3$ and the result of \cite{viana},
the question above about the behaviour of index one minimal surfaces could be stated as follows:

\begin{problem}
\label{problem}
{\it Which compact orientable three-manifolds $M$ admit a Riemannian metric $g$ with positive Ricci curvature and a closed two-sided surface
$\Sigma\subset M$ of genus $3$ such that $\Sigma$ is an index one minimal surface in $(M,g)$? Are there only finitely many of them?
}
\end{problem}
 

We outline the proof of the theorem.
Consider the {\it Euclidean projective space} $RP^3(1/2)$ given as a compact $3$-manifold
diffeomorphic to usual projective space \mbox{${RP}^3\simeq S^3/\pm$} with a singular flat Riemannian metric.
The singular set is a $1$-net with six edges and four vertices, Figure \ref{dunbar}.
It is the quotient space of the cubic \mbox{$3$-torus} $T^3(1)$ under the action of a finite
group of rigid motions, see \cite{dunbar} and Section \ref{I222} below.
By using International Crystallographic Notation \cite{bilbao}, the group is named $I222$ and
it is generated by the $2$-fold screw motions around the principal axes of the cube of side $1/2$ that apply
a perpendicular face into the opposite one,
\[
RP^3(1/2)=T^3(1)/I222.
\]

The proof of Theorem \ref{teor0} depends on the fact that the flat projective space can be approximated by smooth
Riemannian metrics on ${RP}^3$ with positive sectional curvature in terms of the Gromov-Hausdorff distance, the
approximation being in the $C^\infty$ topology outside of any neighborhood of the singular net.

\vspace{.2cm}

In a second step we construct an index one minimal surface in the flat projective space.
To do that we start with the Schwarz $P$ minimal surface, Figure \ref{Schwarz1}.
Marty Ross \cite{ross} showed that this surface has index one in the flat $3$-torus
and that result plays an important role in the history of the second variation of
area for minimal surfaces.
If we think of $RP^3(1/2)$ as a twisted kind of the flat $3$-torus,
then the surface we use is a twisted version of the Schwarz $P$ minimal surface.
By combining Ross' theorem and the symmetry properties of Schwarz' surface, we
show that the new surface has index $1$ and nullity $0$.


The above ingredients  allow us to prove Theorem \ref{teor0} 
by applying an implicit function theorem argument.

 \vspace{.2cm}

\begin{remark}
Note that the Riemannian metrics of Theorem \ref{teor0} admit a second closed minimal surface of index
$1$ and genus $1$ (equal to the Heegaard genus of $RP^3$), \cite{ketovermarquesneves}.
\end{remark}


\section{Preliminaries}

 
Let $\Sigma$ be an orientable closed minimal surface immersed in an orientable Riemannian \mbox{3-manifold} $M$.
We denote by $I\hspace{-.05cm}som(\Sigma)$ the group of transformations on
$\Sigma$ that come from isometries of the ambient space $M$.

The second variation formula of area is given by the {\it index form} $Q(-,-)$ associated to the
{\it Jacobi operator} $L=\Delta + Ric(N)+|A|^2$, where $\Delta$ is the Laplacian on $\Sigma$,
$Ric(N)$ is the Ricci curvature of $M$ along the unit normal vector $N$ of the surface and $|A|^2$ is the square of the
norm of the second fundamental $A$ of the immersion.
If $\varphi\in C^2(\Sigma)$ is an smooth function, then

\[
Q(\varphi,\varphi)=\int_\Sigma |\nabla\varphi|^2-\left(Ric(N)+|A|^2\right)\varphi^2 =
-\int_\Sigma \varphi L \varphi \ dA.
\]

The eigenvalues of $L$ are noted as $\lambda_0 < \lambda_1\cdots < \lambda_k\rightarrow \infty$, each one with 
finite multiplicity $m_k$.
The eigenspace $V_k(\Sigma)$ is the $m_k$-dimensional space of $\lambda_k$-eigenfunctions given
by the functions $u$ such that
\[
L u +\lambda_k u =0 \hspace{.2cm} on \hspace{.2cm}\Sigma.
\]

The eigenvalue $\lambda_0$ is of multiplicity one and the eigenspace is
generated by a positive function $\varphi_0$.
In particular, $\varphi_0$ is invariant under any isometry
in $I\hspace{-.05cm}som(\Sigma)$.
The {\it Jacobi functions} are the eigenfunctions with an eigenvalue equal to $0$
(if any) and the {\it nullity} of $\Sigma$ is the multiplicity of the $0$-eigenvalue.
The immersion is said to be {\it stable} when $\lambda_0\geq 0$ and the {\it index} of $\Sigma$ is the
number of negative eigenvalues (counted with multiplicities).

 
We will consider embedded non totally
geodesic two-sided minimal surfaces $\Sigma\subset M$
in a compact orientable ambient space with nonnegative Ricci curvature.
From the second variation formula we get $\lambda_0<0$ and so, $\Sigma$ is unstable.
We focus on the {\it index one} case (or equivalently $\lambda_1\geq 0$).

\vspace{.2cm}
 
We will also consider the second variation formula for fixed boundary.
Let $D\subset \Sigma$ be a domain with piecewise smooth boundary $\partial D$.
The eigenvalues of $L$ will be denoted as
$\mu_1 < \mu_2 < \cdots$ and the corresponding eigenfunctions satisfy the equation
\[
Lv+\mu_k v =0 \, \, \, {\rm on} \, \, \, D \, \, \, \, \, {\rm and }\, \, \,\, \, v=0 \, \, \, {\rm on} \, \, \,
\partial D.
\]

The first eigenfunction $\varphi_1$ is positive and if $\mu_1 >0$ then we say that $D$ is
{\it strictly stable (for the fixed boundary problem)}.

\vspace{.2cm}

Let $M$ be a compact orientable $3$-manifold and $\Sigma\subset M$ a closed two-sided surface.
The surface is said to be a {\it Heegaard surface} if it separates $M$ in two handle bodies
$M\hspace{-.1cm}-\hspace{-.1cm}\Sigma=\Omega \cup\Omega'$.
The {\it Heegaard genus} of $M$ is the minimal genus of its Heegaard surfaces.
We say that the Heegaard surface is {\it strongly irreducible} if any two proper essential
discs $\triangle\subset \Omega$ and $\triangle'\subset \Omega'$ at different sides of $\Sigma$
have nondisjoint boundaries.
If $M$ has positive Ricci curvature
$Ric>0$ and $\Sigma$ is a minimal surface, then Lawson \cite{lawson} proved that $\Sigma$ is a Heegaard
surface. The same result holds for non totally geodesic minimal surfaces in a flat $3$-manifold, Meeks \cite{meeks}.


\section{The crystallographic groups $I222$ and $Immm$}
\label{I222}

We consider the solid $3$-cube of side $a>0$
\[
C(a)= \{(x,y,z)\in \mathbb{R}^3\, / \, 0\leq x,y,z \leq a \}.
\] 
Its principal axes are the straight lines passing through the center of the cube $(a/2, a/2, a/2)$ 
and parallel to the coordinate axes. If we identify opposite faces of $C(a)$ by the translations
\[
(x,y,0) \mapsto (x,y,a) \hspace{1cm} (x,0,z) \mapsto (x,a,z)
\hspace{1cm} (0,y,z) \mapsto (a,y,z)
\]
we obtain the {\it cubic} $3$-torus $T^3(a)$ with the flat Riemannian metric. 
Equivalently, the flat torus is obtained as the quotient of the Euclidean $3$-space 
by the lattice $\Gamma_a$ generated by the vectors $(a,0,0), (0,a,0)$ and $(0,0,a)$, 
$
T^3(a)=\mathbb{R}^3/\Gamma_a
$ 
and $C(a)$ 
 is a fundamental region of the torus.

\vspace{.2cm}

We will consider in particular the cases $a=1$ and $a=1/2$.

We have the $8\hspace{-.1cm}:\hspace{-.15cm}1$ covering map
\[
\Pi:T^3(1)\longrightarrow T^3(1/2).
\]
 
\vspace{.2cm}

We will also consider the tiling of $T^3(1)$ by eight cubes of edge $1/2$,
\begin{equation}
T^3(1)=\bigcup_{\alpha_i=0,1/2} C(1/2)+(\alpha_1,\alpha_2,\alpha_3). 
\label{tiling21}
\end{equation}
In this case $C(1/2)$ is a fundamental region of that tiling, see Figure \ref{fig-tiling21} (a).
\begin{figure}[h]
\begin{center}
\includegraphics[width=7cm]{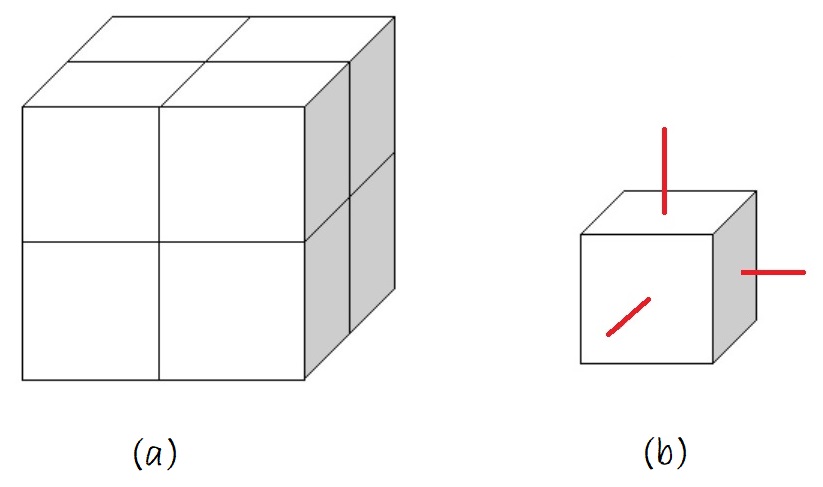}
\caption{}
\label{fig-tiling21}
\end{center}
\end{figure}

\vspace{.5cm}

 Throughout this paper, by a {\it crystallographic group} $\mathcal G$ we mean a finite subgroup of 
the isometry group of the cubic $3$-torus $T^3(1)$, ${\mathcal G} \subset Isom(T^3(1))$, 
see the Bilbao Crystallographic Server \cite{bilbao}.
Following \cite{bilbao}, to describe rigid motions of  $T^3(1)$ we will simply write the 
image of a general point $(x,y,z)$ of $\mathbb{R}^3$, the quotient by the lattice must
be understood. The torus $T^3(1)$ is the quotient by the translations 
$(x+1,y,z),(x,y+1,z),(x,y,z+1)$.

\vspace{.2cm}

Let us consider the three order-$2$ screw motions whose screw axes are the principal axes 
of $C(1/2)$, one of which transforms the face $x=0$ into the opposite one $x=1/2$, and 
the same for the two other directions, Figure \ref{fig-tiling21} (b),\begin{equation}
\left\{
\begin{array}{l}
(-x+1/2,-y+1/2,z+1/2) \\
(-x+1/2,y+1/2,-z+1/2) \\
(x+1/2,-y+1/2,-z+1/2).
\end{array}
\right.
\label{21-screw}
\end{equation} 

The screw-motions (\ref{21-screw}) generate the group $I222$ of order $8$ which, viewed as a transformation group on $T^3(1)$, 
is given by
\begin{equation}
I222 =
\left\{
\begin{array}{l}
(0,0,0)\hspace{-.05cm}+ \hspace{.3cm} (1/2,1/2,1/2) + \vspace{.1cm}\\
(x,y,z) \hspace{.35cm}	(-x,-y,z) \hspace{.35cm} (-x,y,-z) \hspace{.38cm}  (x,-y,-z).
\end{array}
\right.
\label{i222}
\end{equation}

The rigid motions of the group consists of the transformations in the second row plus any 
of the translations given in the first row. Its elements are as follows

\begin{itemize}
  \item[$\circ$] the identity map,
  \item[$\circ$] the three screw motions given by equation (\ref{21-screw}),
  \item[$\circ$] three axial symmetries around  the edges of $C(1/2)$ (composition of two screw motions). 
  Each one of these symmetries has four parallel axes in the torus $T^3(1)$.
  \item[$\circ$] The central translation $(x+1/2,y+1/2,z+1/2)$ (composition of the three screw motions).
\end{itemize}

\vspace{.2cm}

If we add the mirror symmetry $(x,y,-z)$ to the screw motions (\ref{21-screw}), then the generated crystallographic group $Immm$  
is given by the sixteen transformations of $T^3(1)$, see  \cite{bilbao},
\begin{equation}
Immm\, = \,
\left\{
\hspace{-.1cm}
\begin{array}{llll}
{\color{white}.} (0,0,0)+ & (1/2,1/2,1/2)+
\\
(\pm x, \pm y, \pm z).
\end{array}
\right.
\label{immm}
\end{equation} 
The group $Immm$
is formed by eight orientation preserving motions $Immm^+=I222$ and eight orientation reversing ones, 
$Immm^-$. So $I222$ is a subgroup of $Immm$, the quotient is the cyclic group of order two $Immm/I222=\mathbb{Z}_2$ 
and $Immm^-$ is a coset of that quotient.

\vspace{.2cm}

The rigid motions of $Immm^-$ are described as follows:
\begin{itemize}
  \item[$\circ$] Two central symmetries:
  \begin{itemize}
 \item[$\cdot$] one centered at the eight vertices of the cube $C(1/2)\subset T^3(1)$,
\item[$\cdot$] and the other centered at the middle points of the eight cubes of the 
 tiling (\ref{tiling21}). 
\end{itemize}
 \item[$\circ$] Three planar reflections at the faces of the cube $C(1/2)$.
The mirror planes of each one of these reflections are given by
\begin{itemize}
  \item[$\cdot$]  $x=0$ and $x=1/2$, 
  \item[$\cdot$] $y=0$ and $y=1/2$ and 
  \item[$\cdot$] $z=0$ and $z=1/2$.  
\end{itemize} 
\item[$\circ$] Three glide reflections (no fixed points).
\end{itemize}
Note that the reflections $(-x,y,z)$ and $(1-x,y,z)$ are different in $\mathbb{R}^3$ but they coincide
 if viewed as transformations of the $3$-torus $T^3(1)$.

\vspace{.2cm}

Another system of generators of $Immm$ is given by the following rigid motions of $\mathbb{R}^3$: 
\begin{equation}
\label{genera1}
\left\{
\begin{array}{l}
\mbox{six mirror symmetries with respect to the faces of $C(1/2)$ }
\\
\mbox{and the central symmetry with respect to its center.}
 \end{array}
\right.
\end{equation}


\subsection{The Euclidean projective space}
If the crystallographic group $\mathcal G\subset Isom(T^3(1))$ consists of orientation preserving rigid motions, then the map 
$T^3(1) \longrightarrow T(1)^3/\mathcal{G}$ is called an (orientable) 
{\it Euclidean orbifold.} The quotient space $T(1)^3/\mathcal{G}$ admits the structure of a closed orientable \mbox{$3$-manifold}, 
called the {\it underlying space}, endowed with a flat riemannian metric with singularities. 
The {\it singular set} is the $1$-cycle given by the image of fixed points of elements 
of $\mathcal G-\{Identity\}$ and the quotient map $T^3\longrightarrow T^3(1)/{\mathcal G}$ is a branched riemannian covering. 
For a presentation of Euclidean orbifolds see Dunbar \cite{dunbar}.
  
\vspace{.2cm}

The  real projective space $RP^3$ is the manifold obtained by identifying antipodal points of 
the boundary sphere of the Euclidean $3$-ball. Following a related point of view, let us 
introduce the {\it Euclidean projective space} $RP^3(1/2)$  the (Alexandrov) metric space 
obtained from the region $C(1/2)$ by identifying each face of the cube with the opposite one 
by using the equivalent relation 
\begin{equation}
\sim \hspace{.2cm}=
\left\{
\begin{array}{c}
(x,y,0) \sim (1/2-x,1/2-y,1/2) \vspace{.1cm}\\
(x,0,z) \sim (1/2-x,1/2,1/2-z) \vspace{.1cm}\\ 
(0,y,z) \sim (1/2,1/2-y,1/2-z)
\end{array}
\right.
\hspace{1cm}
RP^3(1/2)=C(1/2)/{\hspace{-.06cm}\sim},
\label{3-screw}
\end{equation}
see Figure \ref{fig-tiling21} (b). From (\ref{21-screw}) and (\ref{3-screw}) we have that 
the Euclidean projective space is the {\it orbifold} space given as the quotient 
\begin{equation}
RP^3(1/2)= T^3(1)/I222, \hspace{1cm}  \Pi^*:T^3(1)\longrightarrow RP^3(1/2). 
\label{Pi}
\end{equation}
The projection map $\Pi^*$ is $8$ to $1$ and $RP^3(1/2)$ is a flat riemannian $3$-manifold with singularities. 
The singular set is the image by $\Pi^*$ of the edges of the cube $C(1/2)\subset T^3(1)$.

\begin{figure}[h]
\begin{center}
\includegraphics[width=4cm]{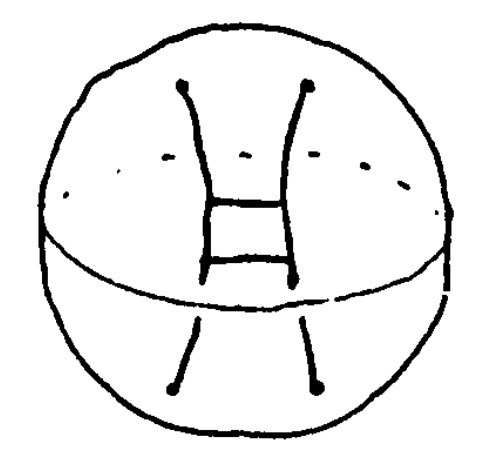}
\caption{ Picture of the Euclidean orbifold $T^3(1)/I222$ appearing in the PhD thesis of Dunbar 
\cite{dunbar}, p 51. The underlying space is the projective space given as an Euclidean $3$-ball with antipodal points identified. 
The $1$-cycle inside the ball represents the singular set of the projective manifold.}
\label{dunbar}
\end{center}
\end{figure}

\vspace{.2cm}

\section{The Schwarz' $P$ minimal surface}
\label{twisted1}

\begin{figure}[h]
\begin{center}
\includegraphics[width=6cm]{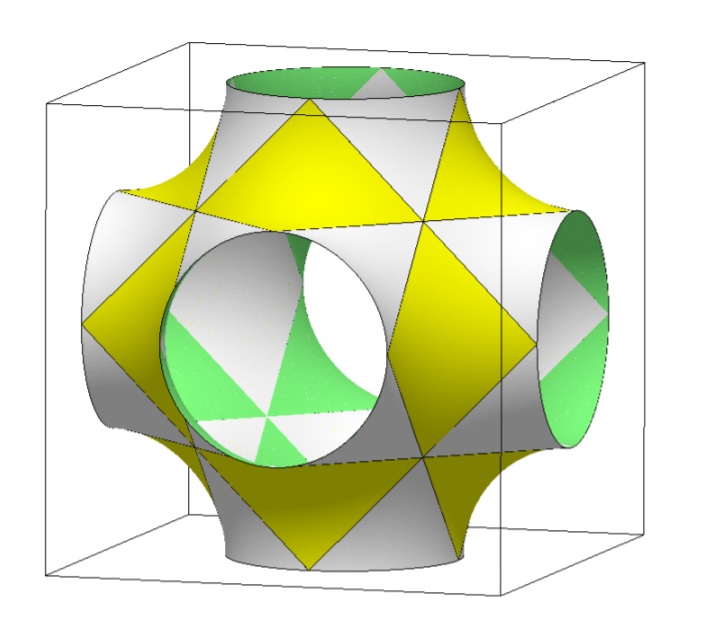}
\caption{The fundamental piece $S$ of Schwarz' $P$ minimal surface in $C(1/2)$ has genus $0$ and 
is orthogonal to the boundary of the cube along six congruent convex curves. 
So, it is a free boundary minimal surface in the cube and identifying the opposite faces
we obtain the Schwarz P minimal surface in the 3-torus $\Sigma\subset T^3(1/2)$.}
\label{Schwarz1}
\end{center}
\end{figure}
The {\it Schwarz' $P$ minimal surface} $\Sigma\subset T^3(1/2)$ is a closed orientable minimal surface of genus $3$  embedded in 
the cubic $3$-torus, see Figure \ref{Schwarz1}, which consists of $24$ congruent skew minimal quadrilaterals. It  contains $12$ closed 
straight lines parallel to the plane coordinates forming a net with $8$ triple vertices and $12$ double ones.  

\vspace{.2cm}

The Gauss map $N:\Sigma\longrightarrow S^2(1)$ is a conformal map of degree $2$ whose branch 
values are the eight vertices of the cube inscribed in the unit sphere  $a(\pm 1, \pm 1, \pm 1)$ with $a=1/{\sqrt{3}}$. 
In particular, $\Sigma$ is a hyperelliptic Riemann surface. 

\vspace{.2cm}

The isometry group of the $P$ surface is $I\hspace{-.05cm}som(\Sigma)=Im\overline{3}m$ of order $96$, and it 
is obtained by composition of the following rigid motions: 

\begin{itemize}
\item[$a)$] the translations 
$(1/2,0,0),(0,1/2,0),(0,0,1/2)$ generating the simple cubic lattice, 
\item[$b)$] the $48$ isometries of the group of the cube $m\overline{3}m$ and
\item[$c)$] the body centered translation $(x+1/4,y+1/4,z+1/4)$.
\end{itemize}

The subgroup of {\it side preserving} isometries of $\Sigma$,  i.e. that preserve the 
unit normal vector $N$ or, equivalently, the components of  $T^3(1/2)-\Sigma$,  
is $I\hspace{-.05cm}som(\Sigma)^+=Pm\overline{3}m$ of order $48$ and is generated by the transformations in $a)$ and $b)$.

The body centered translation $c)$ interchanges these components and belongs to the class of side reversing isometries 
$I\hspace{-.05cm}som(\Sigma)^-$. The groups under consideration verify the relationship
\begin{equation}
I222\subset Immm\subset Pm\overline{3}m\subset Im\overline{3}m.
\label{grupos}
\end{equation}

\vspace{.3cm}

 Ross showed the following interesting property of the Schwarz $P$ minimal surface.

\begin{theorem}[Ross \cite{ross}]  
 The  Schwarz' $P$ minimal surface $\Sigma\subset T^3(1/2)$  has index $1$ and \mbox{nullity $3$.}
\label{Psurface-ross} 
\end{theorem}


In particular, the only Jacobi functions on $\Sigma$ are the linear functions of its Gauss map. 
 Some other classical periodic minimal surfaces have index one too, see Ejiri and Shoda 
 \cite{ejirishoda}. The known proofs depend (at some moment) of explicit calculations. For more about 
 index one minimal surfaces in flat $3$-manifolds see \cite{ritoreros1,ritore,ros1}.

\vspace{.5cm}

\subsection{The Square Catenoid.}

The planar 2-tori $z=1/8$ and $z=3/8$ divide $T^3(1/2)$ in two connected components. The part of $\Sigma$ in each one of 
these components is called the {\it Square Catenoid}. It is an annulus bounded by a pair of squares lying in the above horizontal
 planes and at intermediate heights the horizontal sections are convex curves, see Figure \ref{XsquarecatenoidX},
\[
  Cat=\Sigma\cap\{1/8\leq z\leq 3/8\} ,\hspace{1cm} Cat'=\Sigma\cap\{3/8\leq z\leq 5/8\}. 
\]
 Moreover, both annuli are congruent and the centered traslation by the vector $(1/4,1/4,1/4)$
transforms one into the other,
\[
Cat + (1/4,1/4,1/4)=Cat' \hspace{.6cm} and \hspace{.6cm} Cat \cup Cat' = \Sigma.
\]

\begin{figure}[h]
\begin{center}
\includegraphics[width=4cm]{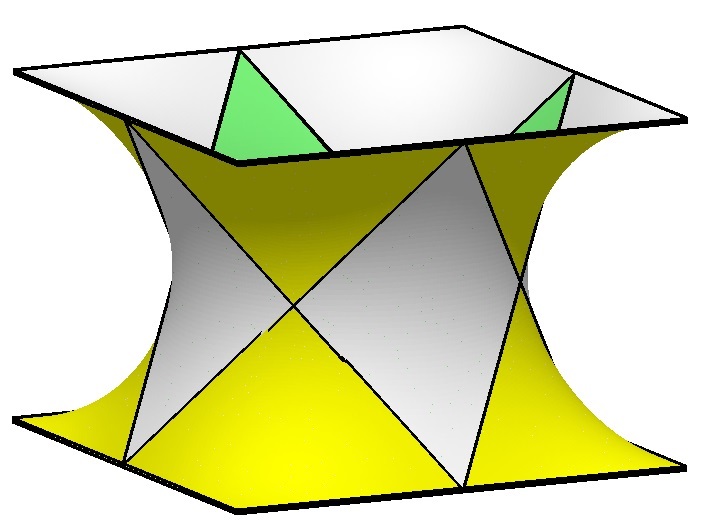}
\caption{The Square Catenoid is a minimal annulus bounded by two parallel squares of side $\sqrt{2}/4$ which differ by the 
vertical translation $(0,0,1/4)$. It decomposes into twelve congruent minimal quadrilaterals and the Schwarz' $P$ surface is 
the union of two square catenoids.}
\label{XsquarecatenoidX}
\end{center}
\end{figure}

\vspace{.1cm}

Although we will not give a proof of Theorem \ref{Psurface-ross}, we show the following consequence, which in fact takes 
part of the inside arguments of \cite{ross} and is equivalent to the theorem itself.
 
\vspace{.3cm}

\begin{corollary} The Square Catenoid is strictly stable (for the fixed boundary problem).
\label{corollarysquarecatenoid}
\end{corollary}
\begin{proof}
Let $v:Cat\longrightarrow \mathbb{R}$ be the first eigenfunction of the Square Catenoid 
 \[
Lv+\mu_1 v=0 \, \, on \,\, Cat \hspace{.6cm} and  \hspace{.6cm} v=0 \,\, at \,\, \partial Cat
\]
We know that $v$ is positive in the interior of $Cat$ and if we extend $v$ as $0$ on $Cat'$ we get a nonzero piecewise smooth function 
$v:\Sigma\longrightarrow \mathbb{R}$. If we repeat the same procedure by changing the role of $Cat$ and $Cat'$ we obtain a second 
piecewise function $v'$ on $\Sigma$ with support equal to $Cat'$ such that 
\[ 
\int_\Sigma v v'=0  \hspace{1cm} Q(v,v')=0.
\]
As $v$ and $v'$ are linearly independent, if we denote by $U$ the plane generated by $v$ and $v'$, then for any $u\in U$ we have
\begin{equation}
\label{corolar}
Q(u,u)= \mu_1\int_\Sigma u^2.
\end{equation}

Reasoning by contradiction, assume that $\mu_1\leq 0$. Let us take $u \in U-\{0\}$ orthogonal to the 
$\lambda_0$ - eigenfunction $\varphi_0$ on $\Sigma$, i. e. 
\[
\int_\Sigma u\varphi_0 =0.
\]
As $\Sigma$ has index $1$, from (\ref{corolar}) we conclude that $\mu_1=0$ and so $u$ is a Jacobi function. The nullity of 
$\Sigma$ is equal to $3$ and so it follows that $u$ is linear function of the Gauss map $N$. 
At the boundary of $Cat$ we have $u=0$. This implies that $u$ is identically cero and this contradiction gives $\mu_1>0$. 
Therefore the Square Catenoid is strictly stable. 
\end{proof}

\vspace{.3cm}


\subsection{The Twisted $P$ minimal surface}

Let us start with the covering $\Pi:T^3(1)\longrightarrow T^3(1/2)$ in (\ref{Pi}) and the Schwarz' P minimal surface 
 \mbox{$ \Sigma\subset T^3(1/2)$}.  We consider the pullback image and the restricted 
 $8\hspace{-.1cm}:\hspace{-.12cm}1$ covering map 
\[
\Sigma^{-1}=\Pi^{-1}(\Sigma)\subset T^3(1)
\hspace{1.3cm} \Pi:\Sigma^{-1}\longrightarrow \Sigma.
\]
The closed minimal surface $\Sigma^{-1}$ has genus $17$, see Figure \ref{pdeschwarz22}. 

\vspace{.2cm}

From now on, the transformation groups $I222$ and $Immm$ and the coset \mbox{$Immm^-$}, 
will be viewed as acting either on the unit $3$-torus or on the surface $\Sigma^{-1}\subset T^3(1)$, depending on the case.

\vspace{.2cm}
 
We define the {\it twisted $P$  minimal surface} $\Sigma^*$ in the flat projective space $RP^3(1/2)$ 
by taking the quotient by the group $I222$,
\[
\Sigma^*= \Sigma^{-1}/I222 \hspace{1cm} RP^3(1/2) = T^3(1)/I222 \hspace{1cm}\Sigma^*\subset RP^3(1/2).
\]
The quotient surface $\Sigma^*$ is a compact, orientable embedded minimal surface with $genus(\Sigma^*)=3$
that does not meet the singular set of $RP^3(1/2)$. We will also consider the 
$8\hspace{-.06cm}:\hspace{-.1cm}1$ covering maps 
\[
\Pi^*:T^3(1)\longrightarrow RP^3(1/2)  \hspace{2cm} \Pi^{*}:\Sigma^{-1}\longrightarrow \Sigma^*
\]
with covering transformations group $I222$. Note that $\Pi^*$ is a branched covering between the 
\mbox{$3$-torus} and the real projective space and when restricted to the 
surfaces $\Sigma^{-1}$ and $\Sigma^*$, it defines a (unbranched) covering map.

\vspace{.2cm}

\begin{figure}[h]
\begin{center}
\includegraphics[width=5.5cm]{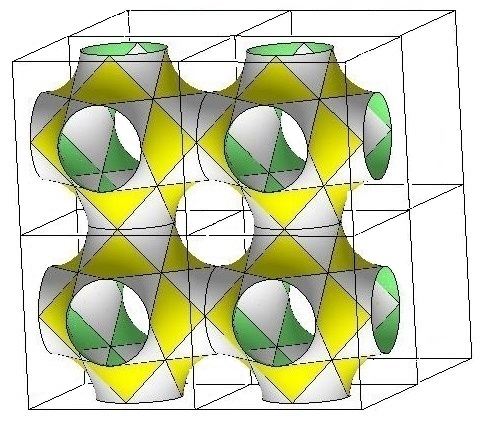}
\caption{The fundamental domain of the surface $\Sigma^{-1}$ on the cube of side $1$ is formed 
by eight copies of the fundamental piece $S$  of the Schwarz surface $\Sigma$.}
\label{pdeschwarz22}
\end{center}
\end{figure}

 \begin{figure}[h]
\begin{center}
\includegraphics[width=4cm]{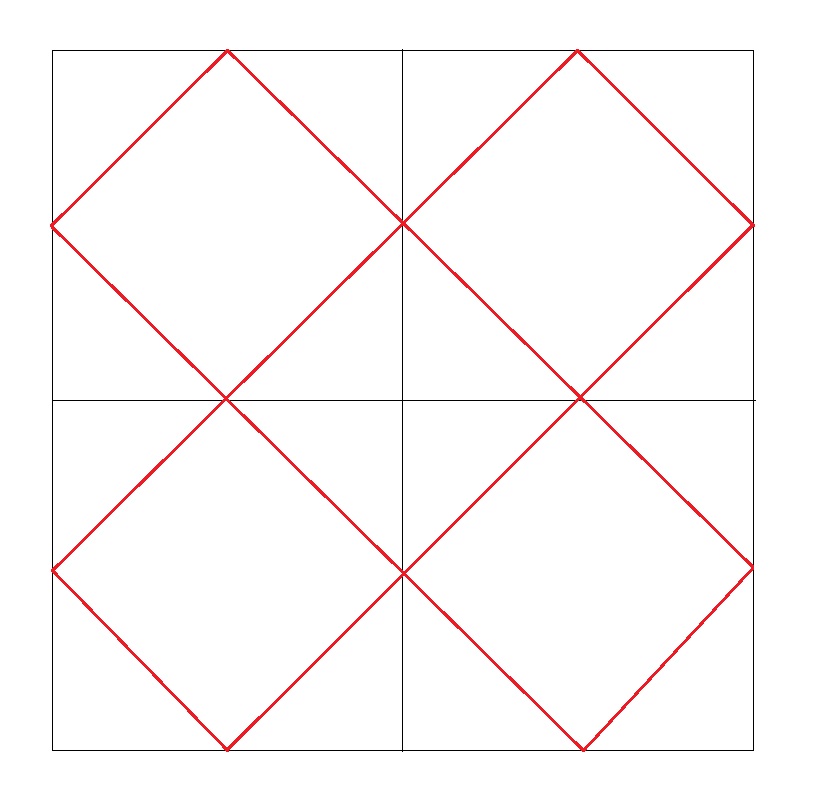}
\caption{For later use it will be of interest to retain the set of straight lines 
in the Schwarz P surface $\Sigma^{-1} \subset T^3(1)$.
Each of these lines is parallel to one of the coordinate planes 
and the figure shows the straight lines (in red) at height $z=1/8$ in the cube $C(1)$.}
\label{rectas}
\end{center}
\end{figure}

Another way of thinking about the twisted surface is as follows.
Identifying the boundary of the cube $C(1/2)$ by the translations vectors
$(1/2,0,0)$, $(0,1/2,0)$,$(0,0,1/2)$, the fundamental piece $S$ gives the closed {\it Schwarz' $P$ minimal surface} 
in the cubic $3$-torus $\Sigma\subset T^3(1/2)$, see Figure \ref{Schwarz1}. 
However, if we identify the opposite faces of the cube by using the screw motions in (\ref{3-screw}) we obtain the 
{\it twisted $P$ minimal surface} of genus $3$ in the flat projective space $\Sigma^*\subset RP^3(1/2)$.

\vspace{.2cm}

In this subsection we will show one of the central results of the paper: The twisted Schwarz P surface has index 1 and 
nullity 0, Theorem \ref{teortwisted}. If the fundamental piece $S$ of the Schwarz P minimal surface in Figure \ref{Schwarz1} 
had index $1$ for the free boundary problem in the cube $C(1/2)$, then Theorem \ref{teortwisted} will follow as a direct 
consequence. But it is known that free boundary minimal surfaces of the cube with six boundary components cannot have 
index $1$, see Theor. 5 in Ros and Vergasta \cite{vergasta}. This illustrates the need for a more elaborate argument.

\vspace{.2cm}

A Continuous function on the twisted P surface $\Sigma^*$ will be thought of as one on the pullback 
surface $\Sigma^{-1}$, $\Pi^*:\Sigma^{-1}\longrightarrow \Sigma^*$, invariant under the group $I222$
\[
u:\Sigma^{-1}\longrightarrow \mathbb{R}, \hspace{.3cm}  {\rm with} \hspace{.2cm} u\circ \phi=u \hspace{.2cm} \forall \phi\in I222.
\]
We say that a function $u$ of this type is {\it odd} (with respect to the pair $I222\subset Immm$) 
if $u\circ \phi = -u$, $\forall \phi\in Immm^-$.
In the case $u$ is invariant under of all the rigid motions of $Immm$, then we say that $u$ is {\it even}.
As $Immm/I222=\mathbb{Z}_2$, then the space of continuous $I222$-invariant functions on $\Sigma^{-1}$ decomposes 
as a direct sum
\[
\{u\in C^0(\Sigma^{-1})\,/ \, u\circ \phi =u, \, \forall \phi\in I222\}=Odd\oplus Even.
\] 

On the minimal surface $\Sigma^*\subset RP^3(1/2)$ we have the Jacobi operator $L=\Delta +|A|^2$, the eigenvalues 
$\lambda_k^*=\lambda_k(\Sigma^*)$ and the eigenspaces $V_k(\Sigma^*)$. In this section we will see these objects over 
$\Sigma^{-1}$. The  eigenspaces $V_k(\Sigma^*)=\{u:\Sigma^* \longrightarrow \mathbb{R}\, / \, Lu+\lambda_k^* u=0\}$ 
will be identified with the space of $\lambda_k^*$-eigenfunctions on $\Sigma^{-1}$ which are invariant under the 
transformations of $I222$,
\begin{equation}
W_k = 
\{u:\Sigma^{-1}\to \mathbb{R} \, /\, Lu + \lambda_k^* u =0, \ 
u\circ \phi = u\ \forall \phi\in I222  \}.
\label{wk}
\end{equation}
The minimal surface $\Sigma^{-1}$ could have other eigenvalues that do not admit $I222$-invariant eigenfunctions. 
These are not considered here. However, the smallest eigenvalue $\lambda_0^*$ of $\Sigma^*$ is also the smallest eigenvalue 
of $\Sigma$ and $\Sigma^{-1}$, $\lambda_0^*=\lambda_0<0$. This is because, in all the cases, the first eigenfunction is invariant 
under all the isometries of the surface. In particular, the $\lambda_0$-eigenspace on $\Sigma^{-1}$ is generated by a positive 
even eigenfunction $\varphi_0:\Sigma^{-1}\longrightarrow\mathbb{R}$. 
Moreover, the decomposition of $W_k$ into $Odd$ and $Even$ functions holds on the eigenspaces $W_k$, for any $k$, too. 

\vspace{.2cm}

From theorem \ref{Psurface-ross} we have $\lambda_1=\lambda_1(\Sigma)=0$. 
Now we will prove that $\lambda_1^*=\lambda_1(\Sigma^*) > 0$. 
\begin{theorem}
The Twisted $P$ minimal surface $\Sigma^*\subset RP^3(1/2)$ has index $1$ and nullity $0$.
\label{teortwisted}
\end{theorem}
\begin{proof}
The first eigenvalue $\lambda_0^*$ is negative and the eigenspace of $W_0$ is generated by a 
positive function $\varphi_0\in W_0$ invariant under all the isometries of the surface $\Sigma^{-1}$. 

\vspace{.1cm}

The space $W_1$ consists of $\lambda_1^*$-eigenfunctions of $\Sigma^{-1}$ invariant under the group $I222$ and
the statement of the theorem means that $\lambda_1^*>0$.

\vspace{.2cm}

Reasoning by contradiction, suppose that $\lambda_1^* \leq 0$ and then we will obtain the two {\it Claims} bellow.

\vspace{.2cm}

Claim 1.  {\it Every function $u\in W_1$ is invariant under the group $Immm$, i.e. $u$ is even.}

\vspace{.2cm}

In fact, if $u\neq 0$ is odd then $u$ vanishes at the fixed points of all the isometries $\phi \in Immm^{\,-}$.  
Therefore $u$ is equal to zero at the six Jordan curves in $\Sigma\cap \partial C(1/2)$, as it are the fixed points 
of the three mirror symmetries of the group $Immm$, see Figure \ref{Schwarz1}.
After identifying opposite sides of the cube, we obtain a continuous piecewise smooth function 
$\widetilde{u}$ on the Schwarz P surface $\Sigma$ defined as $u$ restricted to $\Sigma\cap C(1/2)=\Sigma^*\cap C(1/2)$. 
Moreover, $u$ and $\varphi_0$ are orthogonal on $\Sigma^*$, so we have 
\[
\int_\Sigma \widetilde{u}\varphi_0 = \int_{\Sigma^*} u\varphi_0 = 0, \hspace{1cm} {\rm and}\hspace{1cm}
Q_{\Sigma}(\widetilde{u},\widetilde{u})=Q_{\Sigma^*}(u,u) \leq 0.
\]
From Theorem \ref{Psurface-ross} we conclude that $\lambda_1^*=0$ and ${u}$ is a linear function of the Gauss 
map of the Schwarz P minimal surface. But none of these linear functions is well defined on the twisted P surface $\Sigma^*$ 
and this contradiction proves the claim.

\vspace{.2cm}

The axial symmetries of the group $I222$ are those whose axes are the edges of the cube $C(1/2)$. 
Hence every eigenfunction $u\in W_1$  is invariant under these involutions. 

Now we will show that if $\lambda_1^*\leq 0$ then $u$ is also invariant under other types of axial symmetries: 
Those given by the straight lines contained in the Schwarz P surface, see Figure \ref{Schwarz1}. 

\vspace{.2cm}

Claim 2. {\it The functions in $W_1$ are invariant under the axial symmetries whose axes are contained in $\Sigma^{-1}$.} 


\vspace{.2cm}

To prove claim 2, suppose on the contrary that there is $u \in W_1-\{0\}$ anti-invariant with respect to a horizontal straight line 
$\mathcal L\subset T^3(1)$ contained in $\Sigma\cap\{z=1/8\}$ (a red line in Figure \ref{rectas}). In particular, 
$u$ vanishes along $\mathcal L$. {\it Claim} $1$  says  that  $u$ is invariant under $Immm$ and
therefore $u$ is symmetric with respect to the faces of the cube $C(1/2)$. It follows that $u$ vanishes at 
all the straight lines in the Schwarz minimal surface that are at that same height than $\mathcal L$: all the red lines in Figure \ref{rectas}. 

\vspace{.2cm}

The function $u$ is also invariant  by the central symmetry at $(1/4,1/4,$
$1/4)\in C(1/2)$ (recall this motion 
belongs to the group $Immm$) and, therefore, the function $u$ vanishes at the horizontal lines in $\Sigma^{-1}$ at height 
$z=3/8$, too. Therefore $u$ gives an eigenfunction for the fixed boundary problem of the Square Catenoid and
Corollary \ref{corollarysquarecatenoid} contradicts that $\lambda_1^* \leq 0$. That proves the claim.

\vspace{.3cm}

Now we continue with the proof of Theorem \ref{teortwisted}. Recall that, reasoning by contradiction,  we are assuming that 
$\lambda_1^*\leq 0$. Consider two straight lines ${\mathcal L},{\mathcal L}'\subset \Sigma^{-1}$ in the Schwarz surface 
such that ${\mathcal L}\subset \{z=1/8\}$ and ${\mathcal L}'={\mathcal L}+(0,0,1/4)\subset \{z=3/8\}$. 
From {\it Claim} $2$, any function $u \in W_1$ is invariant under the axial symmetries with respect to $\mathcal L$ and 
${\mathcal L}'$. So, $u$ is invariant with respect to the translation $(0,0,1/2)$  (the composition of the two reflections above).
 
 \vspace{.2cm}

In the same way, we have that $u$ is invariant under the translations $(1/ 2,0,0), (0,1/2,0)$ and therefore 
it is well-defined on the Schwarz surface $\Sigma \subset T^3(1/2)$. 
From Theorem \ref{Psurface-ross} we have that $\lambda_1^*=0$ and $u$ is a linear function of the Gauss
map of the Schwarz P surface, $u=\langle N, a\rangle$ for some $a\in \mathbb{R}^3-{0}$.

\vspace{.1cm}

 Let $C=\{z=0\}\cap S$ be the convex curve at the bottom of $\partial S$, Figure \ref{Schwarz1}.  
 As the normal vector $N$ is horizontal along $C$ and $u$ is well-defined in both $\Sigma$ and $\Sigma^*$, 
 it follows that $u=0$ on $C$. Then $a$ is a vertical vector. Repeating the argument for the directions $(1,0,0)$ 
 and  $(0,1,0)$ we get $a=0$ and this contradiction concludes the proof. 
 \end{proof}


\section{The main result}
\begin{figure}[h]
\begin{center}
\includegraphics[width=8cm]{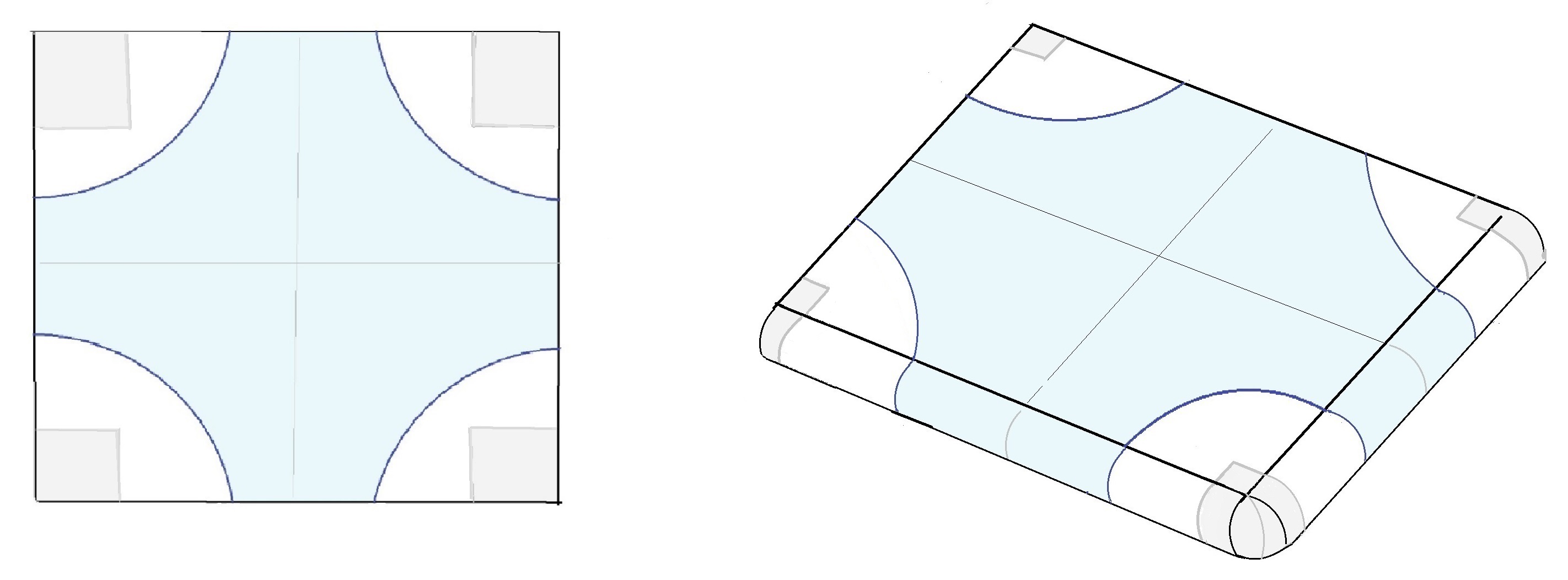}
\caption{Two-dimensional version of the main construction in the proof of Theorem 
\ref{teor0}. The objective is to obtain two blue domains which, after identification, are isometric from 
the intrinsic point of view. One of them lies in the {\it flat projective plane} and the other one
in the usual projective plane  with  a nonnegatively curved metric.  Left: We have 
a blue domain inside the square in the plane meeting its boundary orthogonally.
The sides of the square are identified by the glide motions whose axes are the axes of the figure. 
The generated plane crystallographic group is $pgg$ and the quotient orbifold $T^2/pgg$ is the 
{\it flat projective plane} with two singular points (corresponding to the vertices of the square). 
Right: The planar square has been deformed into a centrally symmetric, smooth convex surface that meets 
the plane $z=0$ orthogonally. The central part remains horizontal and near the edges of the 
square (outside of a neighborhood of its vertices) it has been folded into cylindrical pieces. 
Around the vertices we can deform the initial domain into any convex shape. 
By identifying antipodal points of the boundary, we get a blue domain in the usual projective plane
with a metric of nonnegative curvature. The two constructed blue domains become isometric, as we wanted.
Note that the gluing along the edges is of class $C^1$ at the level of the immersed convex surface, 
but it is of class $C^\infty$ if we just consider the induced flat Riemannian surface structure.}
\label{2D}
\end{center}
\end{figure}

The flat projective space is the quotient of the $3$-torus $T^3(1)$ by the group $I222$
or, equivalently, by identifying the points of the boundary of the cube 
by the order-two screw motions in (\ref{3-screw})
\[
RP^3(1/2) \, \, = \, \, T^3(1)/I222 \, \, = \, \, C(1/2)/{\hspace{-.04cm}\sim}.
\]
To construct an approximation of the {\it singular} flat projective space $RP^3(1/2)$ by nonsingular Riemannian manifolds, 
we consider the hyperplane $\mathbb{R}^3 = \{(x,y,z,t)\in\mathbb{R}^4\, / \, t=0\}$, $\varepsilon  > 0$ and the following  figures:

\begin{itemize}
  \item[$\circ$] the cubes $C(1/2)$ and $C(1/2+\pi\varepsilon)$ in $\mathbb{R}^3$ and the homothety 
  $h:C(1/2)\longrightarrow C(1/2+\pi\varepsilon)$ of ratio $1+ 2\pi\varepsilon$.

\vspace{.12cm}

\item[$\circ$] 
convex ball 
$
B_\varepsilon = \{p\in \mathbb{R}^4 \, / \, dist\big{(}p,C(1/2)\big{)}\leq \varepsilon\}$ in 
$\mathbb{R}^4$ and the top part of its boundary 
$E_\varepsilon = \partial B_\varepsilon \cap \{t\geq 0\}$. 
\end{itemize}

\vspace{.2cm}

$E_\varepsilon$ is a $C^1$ convex graph over 
$\{p\in \mathbb{R}^3\, / \, dist\big{(}p,C(1/2)\big{)}\leq \varepsilon\}\subset\mathbb{R}^3$, its boundary meets orthogonally 
the hyperplane $t=0$ and it consists of the union of the following pieces:
 
\begin{itemize}
  \item[$1)$] 
A flat part parallel to the cube, $C(1/2)+(0,0,0,\varepsilon)$. \vspace{.15cm}
\item[$2)$]
A part made by quarter of cylinders of the type $S^1(\varepsilon) \times \mathbb{R}^2$ following the faces of the cube. 
Each one of these pieces is isometric (as a Riemannian manifold) to a region of $\mathbb{R}^3$. \vspace{.15cm}
\item[$3)$]
Hypercylindrical pieces $\{(x,y,z)\in S^2(\varepsilon)/x,y,z\geq 0\}\times\mathbb{R}$ along the edges of the cube. \vspace{.15cm}
\item[$4)$]
Spherical caps as $\{(x,y,z,t)\in S^3(\varepsilon)\, /\, x,y,z,t\geq 0\}$ centered at the vertices of the cube.
\end{itemize}



\begin{lemma}
\label{lemma}
There are two domains $D\subset C(1/2)$ and $D_\varepsilon \subset E_\varepsilon$ and a dilation map
$\phi:D\longrightarrow D_\varepsilon$ satisfying the following:

\vspace{.1cm}

\begin{itemize}
\item[$i)$] $D$ is equal to $C(1/2)$ minus a small neighborhood of the edges of the cube.

\vspace{.1cm}
\item[$ii)$]
The domain $D_\varepsilon$ is the union of the pieces $1)$ and $2)$ above. 
From the intrinsic point of view, $D_\varepsilon$ is the piecewise smooth $C^\infty$ 
flat Riemannian manifold.
 
\vspace{.1cm}

\item[$iii)$] The ratio of the dilation $\phi:D\longrightarrow D_\varepsilon$ is 
 $1+2\pi\varepsilon$. 
\end{itemize} 
\end{lemma}

\begin{proof}
 The union along the boundary of the halfspace  
$\{(x,y,z,\varepsilon)\in \mathbb{R}^4\, / \, z\geq 0\}$ and the semi cylinder 
$\mathbb{R}^2\times \{(z,t)\in S^1(\varepsilon)\, /\, z\leq 0\}$ give a hypersurface of class $C^1$ in 
$\mathbb{R}^4$. However, from the Riemannian point of view, the union is a $C^\infty$ manifold 
isometric to an open subset of $\mathbb{R}^3$. 

\vspace{.2cm}

As a consequence, the union of the pieces of type $1)$ and $2)$ is a $C^\infty$ flat Riemannian manifold 
$D_\varepsilon$ contained in $E_\varepsilon$. Furthermore we have:

\vspace{.2cm} 
\begin{itemize}
\item[$\cdot$] A domain $D'_\varepsilon \subset C(1/2+\pi\varepsilon)$ and a Riemannian isometry 
{$f:D'_\varepsilon\longrightarrow D_\varepsilon$.}

\vspace{.1cm} 

\item[$\cdot$] the pullback homothetic image $D=h^{-1}(D'_\varepsilon)$, 
$h:D\longrightarrow D'_\varepsilon$, and

\vspace{.1cm}

\item[$\cdot$] the $C^\infty$-dilation map $\phi=h\circ f$ of ratio $1+2\pi\varepsilon$ between flat manifolds
\[
\phi: D \longrightarrow D_\varepsilon \hspace{1cm} D \subset C(1/2)  \hspace{1cm} {\rm and} 
\hspace{1cm} D_\varepsilon \subset E_\varepsilon.
 \]
\end{itemize}
The transformation $\phi$ is the composition of a Euclidean homothety and a Riemannian isometry.
\end{proof}

\vspace{.2cm}

{\it Proof of Theorem\hspace{-.05cm} \ref{teor0}.}
Since $\partial B_\varepsilon$ is formed by pieces of hyperplanes and hypercylinders of the type $1)$, 
$ 2)$, $3)$ y $4)$ above, it is a convex hypersurface of class $C^1$ but not $C^2$.

A compact convex body in $\Omega\subset \mathbb{R}^{4}$ can be approximated 
 by a sequence of smooth convex bodies $\Omega_n$ such that the principal curvatures of $\partial \Omega_n$  
 are positive and tend to the principal curvatures of $\partial\Omega$ in the sense of the surface area
measure, Gruber \cite{gruber} p 133, Weil \cite{weil}. 
For the case of abstract positively curved polyhedral $3$-manifolds we can also use the alternative 
 smoothing result in  \cite{petrunin}.

Therefore, in our case, we have that the convex hypersurface $\partial B_\varepsilon$
can be approximated in the $C^1$ topology by (a centrally symmetric) strictly convex compact hypersurface $S_\varepsilon$
 of class $C^\infty$. 
Moreover, by Lemma \ref{lemma}, the approximation can be assumed of class $C^\infty$ over $D_\varepsilon$. 

\vspace{.1cm}

For small $\varepsilon >0$, if we identify antipodal points of the boundary 
we obtain two metric copies of the $3$-dimensional projective space  
\[
R\hspace{-.03cm}P^3_{\hspace{-.07cm}\varepsilon} =E_\varepsilon /\hspace{-.12cm}\sim \hspace{1cm} {\rm and}  \hspace{1cm}
\overline{R\hspace{-.03cm}P}_{\hspace{-.12cm}\varepsilon}^3=\overline{E}_\varepsilon /\hspace{-.1cm}\sim
\]

that can be taken arbitrarily close in the $C^1$-topology and satisfy the following:

\vspace{.2cm}

$i)$ $R\hspace{-.03cm}P^3_{\hspace{-.07cm}\varepsilon}$ contains an smooth flat domain 
$U_\varepsilon=D_\varepsilon/\hspace{-.14cm}\sim$ obtained from $D_\varepsilon$ by identifying antipodal 
points of $\partial D_\varepsilon$ at level $t=0$. Furthermore this domain contains itself an dilated image of 
the twisted $P$ minimal surface $\Sigma^*_\varepsilon\subset U_\varepsilon$. 

\vspace{.2cm}

$ii)$ $\overline{R\hspace{-.03cm}P}_{\hspace{-.12cm}\varepsilon}^3$ is a $C^\infty$ Riemannian manifold with positive 
sectional curvature containing a domain $\overline{U}_\varepsilon$ as close as we want of $U_\varepsilon$ in the $C^\infty$ topology. 
 
\vspace{.2cm}

From Theorem \ref{teortwisted}, the twisted $P$ minimal surface $\Sigma^*_\varepsilon$ has index $1$ and nullity $0$ and, 
therefore, the theorem follows from the implicit function theorem. 
 $\square$ 

\vspace{.4cm}

The existence of index one closed embedded minimal surfaces in spaces of positive Ricci curvature follows from 
Ketover, Marques and Neves \cite{ketovermarquesneves}. There are minimal genus Heegaard surfaces, $genus(\Sigma)\leq 2$.  
D. Ketover, Y. Liokumovich and A. Song \cite{ketoverliokumovichsong} obtain index one minimal surfaces in compact Riemannian 
$3$-manifolds, without  curvature assumptions, isotopic to a given strongly irreducible Heegaard splitting. These results follow 
Almgren-Pitts \cite{almgren} and Simon-Smith \cite{simon} approaches.
The surfaces can also been constructed by using Min-max theory for Allen-Cahn equation, see Guaraco \cite{guaraco}, 
Chodosh and Mantoulidis \cite{chodoshmantoulidis} and references therein.

\vspace{.2cm}

Although the Schwarz P surface $\Sigma$ in $T^3(1/2)$ is a Heegaard splitting  of minimum genus, the min-max construction in 
this space does not give the P surface but a multiplicity $2$ horizontal $2$-torus. In  the flat projective space $RP^3(1/2)$, 
the Heegaard surface $\Sigma^*$ is not even a minimum genus. 

\vspace{.1cm}

Another remarkable topological property of these surfaces is the following. 
Let $S$ a Heegaard surface in a closed orientable 3-manifold $M$  and  $\Omega,\Omega'\subset M$ be the two 
complementary handlebodies. The Heegaard splitting is said to be {\it strongly irreducible}
 if any pair of essential discs $\Delta\subset\Omega$ and $\Delta'\subset\Omega'$ intersect at their boundary. 
In our situation, neither the Heegaard splitting given by the Schwarz $P$ surface $\Sigma\subset T^3(1/2)$ nor the one given  
by the twisted $P$ surface $\Sigma^*\subset RP^3(1/2)$ are strongly irreducible. In fact, if we consider the {\it inside}/{\it outside} 
regions determined by the $P$ surface in the cube $C(1/2)$, Figure \ref{Schwarz1}, then the disc $\triangle$ given by 
$\{z=0\}\cap Inside$ is an essential disc and  $\{z=1/4\}\cap Outside$ consists in four quadrants that after identifications, 
in both $T^3(1/2)$ and $RP^3(1/2)$, give a second essential disc $\triangle'$  at the other side, the boundary curves being disjoint.

\vspace{.1cm}

As the strong irreducibility  implies the minimal genus property, see e.g. \cite{irreducible}, 
the \mbox{{\it  index\hspace{-.06cm} $1$/genus\hspace{-.06cm}  $3$} } Schoen' question \cite{neves} 
could also be approached, instead of as in Problem \ref{problem}, as follows:

\begin{problem}
{\it Let $(M,g)$ a compact orientable Riemannian 3-manifolds with positive Ricci curvature and $\Sigma\subset M$ be
a closed two-sided minimal surface of index one and genus $3$. Hence $\Sigma$ is a Heegard splitting, Lawson \cite{lawson}, 
but can it be strongly irreducible?}  
\end{problem}

{\footnotesize
\noindent
{Antonio Ros},  {\tt aros@ugr.es},\vspace{-.1cm}\\
Department of Geometry and Topology and 
Institute of Mathematics (IMAG),  \vspace{-.1cm}\\
University of Granada,\vspace{-.1cm}
 18071 Granada, Spain.
}

\end{document}